\documentclass[12pt]{extarticle}

\usepackage[margin=1in]{geometry}
\usepackage{mathtools}
\usepackage{amsthm}
\usepackage{amsfonts,amssymb}
\usepackage{graphicx}
\usepackage{hyperref}
\usepackage{titling}

\theoremstyle{plain}
\newtheorem{theo}{Theorem}[section]
\newtheorem{lemm}[theo]{Lemma}
\newtheorem{prop}[theo]{Proposition}
\newtheorem{coro}[theo]{Corollary}
\theoremstyle{definition}
\newtheorem{defi}[theo]{Definition}
\theoremstyle{remark}
\newtheorem{rema}[theo]{Remark}

\newtheorem{conj}[theo]{Conjecture}

\newcommand{\Omin}{\mathcal{O}_{\textrm{min}}}

\newcommand{\Ominbar}{\overline{\mathcal{O}}_\textrm{min}}

\newcommand{\Cat}{\mathrm{Cat}}
\newcommand{\Wedge}{\scalebox{0.8}{\raisebox{0.4ex}{$\bigwedge$}}}

\newcommand{\slnn}{\mathfrak{sl}_{n+1}}
\newcommand{\slnnd}{\mathfrak{sl}_{n+d+1}}

\newcommand{\Lieg}{\mathfrak{g}}

\newcommand{\Lieh}{\mathfrak{h}}

\newcommand{\Sym}{\mathrm{Sym}}

\DeclareMathOperator{\End}{End}

\DeclareMathOperator{\Proj}{Proj}
\DeclareMathOperator{\Spec}{Spec}

\newcommand{\N}{\mathbb{N}}

\newcommand{\Z}{\mathbb{Z}}
\newcommand{\C}{\mathbb{C}}

\newcommand{\lambar}{{\overline{\lambda}}}

\title{Highest Weight Varieties and Narayana Numbers}
\author{Boming Jia}
\date{}
\begin{document}

\maketitle

\vspace{-5em}
\begin{abstract} 
    We compute the Hilbert series of the coordinate ring of some highest weight varieties. We also explain why Narayana numbers (and their generalizations) appear naturally in the numerator of the Hilbert series of the homogeneous coordinate ring of the Grassmannian $Gr(d,n+d+1)$ and of the minimal nilpotent adjoint orbit in $\mathfrak{sl}_\mathrm{n+1}(\mathbb{C})$.
\end{abstract}

%\noindent\textbf{MSC 2010:} {17B08,17B10,05E10}\\
%\textbf{Keywords:} Hilbert series; Narayana numbers; minimal nilpotent orbits.

\section{Introduction}
The sequence of Catalan numbers $1,2,5,14,42,429,1430,4862,\dots$ is ubiquitous in combinatorics and representation theory.
Defined by the formula
\[
\Cat_n \coloneqq \frac{1}{n+1} \binom{2n}{n},
\]
they count many combinatorial structures, such as triangulations of a convex polygon with \(n+2\) sides and Dyck paths. We refer to \cite{Stanley} for more than 200 interpretations of Catalan numbers.
 Narayana numbers are a refinement of Catalan numbers, which are defined by the formula
$$
N_{n, k}\coloneqq\frac{1}{n}\binom{n}{k}\binom{n}{k+1}.
$$
Note that the sums of the rows in the triangle of Narayana numbers are the Catalan numbers:
\begin{align*}
    \sum_{k=0}^n N_{n,k}
    &= \sum_{k=0}^n\binom{n-1}{k}\binom{n+1}{k+1}-\binom{n}{k}\binom{n}{k+1} \\
    &= \binom{2n}{n}-\binom{2n}{n-1}
    = \frac{1}{n+1}\binom{2n}{n}
    = \Cat_n.
\end{align*}
There are many combinatorial interpretations of the Narayana numbers; for example, $N_{n,k}$ counts $231$-pattern-avoiding permutations in $S_n$ with $k$ descents.

In this paper, we explain how Narayana numbers arise naturally in the context of highest weight varieties. A highest weight variety is the Zariski closure of the orbit of a highest weight vector under the action of a complex simple Lie group. The Hilbert series of a graded algebra encodes the dimensions of its homogeneous components and is a useful tool for studying coordinate rings. We compute Hilbert series for several highest weight varieties and explain why Narayana numbers, and their generalizations, appear in the numerators for the Grassmannian \(Gr(d,n+d+1)\) and the minimal nilpotent adjoint orbit in \(\mathfrak{sl}_{n+1}(\mathbb{C})\).

The paper is organized as follows. In Section $2$, we introduce highest weight varieties and relate the Hilbert series of their coordinate rings to dimensions of irreducible representations. In Section $3$, we compute the Hilbert series of the homogeneous coordinate ring of the Grassmannian $Gr(d,n+d+1)$ and explain the appearance of higher-dimensional Narayana numbers in the numerator. In Section $4$, we compute the Hilbert series of the coordinate ring of the minimal nilpotent orbit in $\mathfrak{sl}_{n+1}$ and show that its numerator consists of Narayana numbers of type $C_n$.

\vspace{2em}
\textbf{Acknowledgments.}
The author was supported by NSFC Grant No.~12225108 and the Shuimu Scholar Program in Tsinghua University. The author thanks Julia Feng for consistent support and encouragement, and thanks Vitaly Bergelson, Victor Ginzburg, Fanhao Kong, Pengcheng Li, Yanpeng Li, Yu Li, Jiang-Hua Lu, Quan Situ, Qingyu Ren, Nikolai Reshetikhin, Peng Shan, and Fang Yang for helpful suggestions and comments.

\section{Highest Weight Varieties.}
We now pass to the representation-theoretic setting needed for highest weight varieties and their Hilbert series.
The results in this section are well known, and we primarily follow Garfinkle's thesis \cite{Garfinkle} with slightly different notations.

Let $\Lieg$ be a complex simple Lie algebra. Let $G$ be a simply connected Lie group with the Lie algebra $\Lieg$. Let $(\pi_\lambda,V_\lambda)$ be an irreducible finite-dimensional representation of $\Lieg$ with non-trivial highest weight $\lambda$. 
Let $V^*_\lambda$ be the dual representation of $\Lieg$. Throughout the section, we identify the symmetric tensor algebra $\Sym(V_\lambda^*)$ of the dual representation $V^*_\lambda$ with the algebra of polynomial functions $\C[V_\lambda]$ on the representation $V_\lambda$.
\begin{defi}
	 Let $v_\lambda\neq0$ be a highest weight vector of $V_\lambda$. We define the highest weight variety 
	$$
	X_\lambda\coloneqq\overline{G v_\lambda},
	$$
    as the Zariski closure of the $G$-orbit of $v_\lambda$,
	and we denote its vanishing ideal by $I\subset\Sym(V_\lambda^*)$.
\end{defi}
\begin{rema} 
    Note that in Chapter III, Section 1 of \cite{Garfinkle}, the ideal $I$ was equivalently defined to be the vanishing ideal of the $G$ orbit of the highest weight vector, and the phrase ``Highest Weight Variety'' was not explicitly addressed.
\end{rema}
Let $v_1=v_\lambda,v_2,\cdots,v_n$ be a basis of $V_\lambda$ consisting of weight vectors, and let $v^*_1=v^*_\lambda,\cdots,v^*_n$ be the dual basis for $V^*_\lambda$ such that $v^*_i(v_j)=\delta_{ij}$ for all $i,j$. Let $\lambar$ be the highest weight for $V^*_\lambda$ so that $V_\lambar= V^*_\lambda$.
For each $k\in\Z_{\geq0}$, we define $$V_{k\overline{\lambda}}\subset \Sym^k(V_\lambda^*)$$ to be the simple $\Lieg$-module generated by $(v^*_\lambda)^k$, and $C_{k\lambar}$ be the unique $\Lieg$-invariant complement of $V_{k\lambar}$ in $\Sym^k(V_\lambda^*)$, so
	$$\Sym^k(V_\lambda^*)=V_{k\lambar}\oplus C_{k\lambar}.$$
Notice that $C_0=C_\lambar=\{0\}$.
\begin{prop}[Proposition III.1 in \cite{Garfinkle}]\label{prop1} The ideal $I\subset\Sym(V_\lambda^*)$ satisfies
	$$I=\bigoplus_{k=2}^{\infty}\,C_{k\lambar}.$$
\end{prop}

\begin{proof}
	
Since $\lambda$ is nontrivial, the orbit $G v_\lambda$ is invariant under the scaling $\C^*$-action on $V_\lambda$. So we have
$$
	I=\bigoplus_{k=0}^{\infty}\,I\cap\Sym^k(V_\lambar).
$$

Let $k\in\Z_{\geq0}$, we denote $I_k\coloneqq I\cap\Sym^k(V_\lambar)$, so it suffices to show that $I_k=C_{k\lambar}$. 

Since $(v^*_1)^k(v_\lambda)\neq0$, we have $(v^*_1)^k\notin I_k$. But $I_k$ is a $\Lieg$-module and $V_{k\lambar}$ is a simple $\Lieg$-module containing the highest weight vector $(v^*_1)^k$, so we have 
$
	V_{k\lambar}\cap I_k=\{0\}.
$
This proves
$$
    I_k\subseteq C_{k\lambar}.
$$

Let $f\in C_{k\lambar}$. Since $f$ is a linear combination of monomials, each of which contains some factors of $v^*_2,\cdots,v^*_n$, we have $f(v_\lambda)=0$. Now we show that $f$ vanishes on $G.v_\lambda$. Since $G$ is simply connected, the $\Lieg$-module $C_{k\lambar}$ is also a $G$-module, so for all $g\in G$,
$$
    f(gv_\lambda)=(g^{-1}f)(v_\lambda)=0.
$$
So we have $f\in I_k$, and this proves
\[C_{k\lambar}\subseteq I_k.\qedhere\]
\end{proof}

Recall the definition and basic geometric setup for Hilbert series, which we will use repeatedly below.
Let $A$ be a graded commutative $\C$-algebra
 $$A=\bigoplus_{i=0}^{\infty}A_i$$
with $A_0=\C$, and $A$ generated by $A_1$, where $n\coloneqq\dim_\C(A_1)<\infty$. Then each $A_i$ is a finite-dimensional $\C$-vector space.
\begin{defi}
    We define the Hilbert series of $A$ as the formal power series
$$
    h_A(t)\coloneqq \sum_{i=0}^{\infty}\dim(A_i)t^i.
$$
\end{defi}
There is a natural surjective homomorphism $\Sym^\bullet A_1\twoheadrightarrow A$ induced by the inclusion $A_1\hookrightarrow A$ and the universal property of $\Sym^\bullet A_1$. Geometrically, this yields an embedding of affine schemes
$$
X\coloneqq\Spec(A)\hookrightarrow \mathbb{A}^n=\Spec(\Sym^\bullet A_1).
$$
The affine scheme $X$ carries a $\mathbb{G}_m$-action given by $z.(a_i)_{i=0}^\infty\coloneqq(z^ia_i)_{i=0}^\infty$, and $X$ is the affine cone over the projective scheme $\Proj(A)\subset \mathbb{P}(\mathbb{A}^n)$.

\begin{coro}\label{Hilbert}
    The homogeneous coordinate ring of the highest weight variety $\C[X_\lambda]$ is isomorphic to a graded algebra
    $
        \bigoplus_{k=0}^{\infty}V_{k\lambar}
    $
    with a $\Lieg$-invariant multiplication. So, the Hilbert series for the graded algebra $\C[X_\lambda]$ is
    $$
        h_{\C[X_\lambda]}(t)=\sum_{k=0}^{\infty}\dim(V_{k\lambar})t^k.
    $$
\end{coro}
\begin{proof}
    Recall that for each $k$, we have $\Sym^k(V_\lambda^*)=V_{k\lambar}\oplus C_{k\lambar}.$ So our statement follows immediately from Proposition \ref{prop1} and the fact that the coordinate ring $$\C[X_\lambda]\cong\Sym(V^*_\lambda)/I$$ as graded algebras.
\end{proof}

\section{The Hilbert Series of $\C_{\mathrm{hom}}[Gr(d,\C^{n+1+d})]$.}
Let $Gr(d,n+d+1)$ be the Grassmannian variety consisting of $d$-dimensional subspaces in $\C^{n+d+1}$. The goal of this section is to compute the Hilbert series of the homogeneous coordinate ring $\C_{\mathrm{hom}}[Gr(d,\C^{n+1+d})]$ of the Grassmannian $Gr(d,n+d+1)$.

There are two generalizations of the Catalan numbers in two directions.
\begin{defi}\label{Cat}
    For any irreducible finite root system $\Phi$ of type $X$ and rank $n$, the Catalan number of type $X_n$ is defined as
$$
    \Cat_{X_n}\coloneqq\prod_{i=1}^n\frac{e_i+h+1}{e_i+1},
$$
where $e_1,\cdots,e_n$ are the exponents of $\Phi$, and $h$ is the Coxeter number of $\Phi$.
\end{defi}
So in type $A_n$, we have $$\Cat_{A_n}=\prod_{i=1}^{n}\frac{i+(n+1)+1}{i+1}=\frac{1}{n+2}\binom{2n+2}{n+1}=\Cat_{n+1}.$$
Similarly, Narayana numbers of $A$-type are
$$
    N_{A_n}(k)\coloneqq N_{n+1,k}=\frac{1}{n+1}\binom{n+1}{k}\binom{n+1}{k+1}.
$$
\begin{defi} Let $d\geq2$.
    The $d$-dimensional Catalan numbers \cite{Sulanke} are defined as
$$C_{d,n}\coloneqq(dn)!\prod_{i=0}^{d-1}\frac{i!}{(n+i)!},$$ 
so that the original Catalan number $C_n$ becomes a special case $C_{2,n}$.
    Similarly, for $0\leq k\leq(d-1)(n-1)$ the $d$-dimensional Narayana numbers \cite{Sulanke} can be defined as
$$
    N_{d,n,k}\coloneqq\sum_{j=0}^k(-1)^{k-j}\binom{dn+1}{k-j}\prod_{i=0}^{d-1}\binom{n+i+j}{n}\binom{n+i}{n}^{-1}.
$$
    And we define $N_{A_n}(d,k)\coloneqq N_{d,n+1,k},$ where $0\leq k\leq (d-1)n$.
\end{defi}

Take $\Lieg=\slnnd(\C)$. We fix the standard choice of simple roots $\Pi=\{\alpha_1,\alpha_2,\cdots,\alpha_{n+d}\}$ and let $\omega_1,\cdots,\omega_{n+d}$ denote the corresponding fundamental weights.
Let $\rho$ be the half sum of the positive roots $\Phi^+\subset\Lieh^*$. 
The affine cone variety $X(d,n+d+1)$ (under the Pl\"ucker embedding) of $Gr(d,n+d+1)$ is identified with the decomposable elements in $\Wedge^d(\C^{n+d+1})$, which is isomorphic to the highest weight variety $X_{\omega_d}$ for the $\slnnd$-representation $V_{\omega_d}=\Wedge^d(\C^{n+d+1})$.

\begin{lemm}\label{dimrep3}
    Let $k\in\N$. The irreducible $\slnnd$-representation $V_{k\omega_{n+1}}$ has dimension
    $$\dim(V_{k\omega_{n+1}})=\displaystyle\sum_{i=0}^{k}N_{A_n}(d,i)\binom{d(n+1)+k-i}{d(n+1)}.$$
\end{lemm}
\begin{proof}
Recall that $\rho=\omega_1+\cdots+\omega_{n+d}$. Let $k\in\N$. By Weyl's dimension formula
\begin{align}\label{Vkomega}
   &\quad \dim(V_{k\omega_{n+1}}) =\prod_{\alpha\in\Phi^+}\frac{\langle k\omega_{n+1} +\rho,\alpha\rangle}{\langle\rho,\alpha\rangle}
   =\prod_{\alpha\in\Phi^+}\frac{\langle k\omega_{n+1}+(\omega_1+\cdots+\omega_{n+d}),\alpha\rangle}{\langle\omega_1+\cdots+\omega_{n+d},\alpha\rangle}\nonumber\\
   & =\prod_{i=0}^{d-1}\prod_{j=0}^{n}\frac{k+(n+1)+i-j}{(n+1)-j}\
   =\prod_{i=0}^{d-1}\binom{k+(n+1)+i}{n+1}\binom{(n+1)+i}{n+1}^{-1}.
\end{align}
Now, by Proposition 4 in \cite{Sulanke}, we have
\[
    \prod_{i=0}^{d-1}\binom{k+(n+1)+i}{n+1}\binom{(n+1)+i}{n+1}^{-1}=\sum_{i=0}^{k}N_{A_n}(d,i)\binom{d(n+1)+k-i}{d(n+1)}.\qedhere
\]

\end{proof}

Now we can reprove Braun's result on the Hilbert series of Grassmannians.
\begin{theo}[Theorem 1 in \cite{Braun}]\label{Gr} 
    The Hilbert series for the homogeneous coordinate ring $\C_{\textrm{hom}}[Gr(d,n+d+1)]$ is
    $$
        h_{\C_{\textrm{hom}}[Gr(d,n+d+1)]}(t)=\frac{\sum_{\ i=0}^{(d-1)n} N_{A_n}(d,i)\,t^i}{(1-t)^{d(n+1)+1}},
    $$
where $N_{A_n}(d,i)$ are the $d$-dimensional Narayana numbers of type-$A$.
\end{theo}
\begin{proof}
Recall that the affine cone variety $X(d,n+d+1)$ of the Grassmannian $Gr(d,n+d+1)$ is isomorphic to the highest weight variety $X_{\omega_2}$ for the $\slnnd$-representation $V_{\omega_d}=\Wedge^d(\C^{n+d+1})$. Since $\Wedge^d(\C^{n+d+1})^*\cong\Wedge^{n+1}(\C^{n+d+1})$, we have $\overline{\omega_d}=\omega_{n+1}$. So by Corollary \ref{Hilbert} and Lemma \ref{dimrep3} we have
    \begin{align*}
        h_{\C[X(d,n+d+1)]}(t)& =\sum_{k=0}^{\infty}\dim(V_{k\omega_{n+1}})t^k
        =\sum_{k=0}^{\infty}\sum_{i=0}^{k}N_{A_n}(d,i)\binom{d(n+1)+k-i}{k-i}t^k\\
        & =\sum_{i=0}^{(d-1)n}N_{A_n}(d,i)\,t^i\ \sum_{j=0}^{\infty}\binom{d(n+1)+j}{j}t^{j} =\frac{\displaystyle\sum_{i=0}^{(d-1)n}N_{A_n}(d,i)t^i}{(1-t)^{d(n+1)+1}}.\qedhere
    \end{align*}
\end{proof}

\begin{coro}[Mukai's Theorem]\label{ThmGrMukai}
    The Hilbert series for the homogeneous coordinate ring $\C_{\textrm{hom}}[Gr(2,n+3)]$ is
    $$
        h_{\C_{\textrm{hom}}[Gr(2,n+3)]}(t)=\frac{\sum_{i=0}^{n}\frac{1}{n+1}\binom{n+1}{i}\binom{n+1}{i+1}\,t^i}{(1-t)^{2n+3}}.
    $$
    Note that the coefficients of the numerator are precisely the Narayana numbers of type $A_n$.
\end{coro}
\begin{rema}\label{remGrMukai}\sloppy
    Since the homogeneous coordinate ring of the Grassmannian $Gr(2,n+3)$ is well known to have a cluster algebra structure of type $A$, the number of clusters is always a Catalan number, which is the sum of the Narayana numbers in the same row (as appears in the numerator of the RHS of the above result). It is interesting to study further the relations between the above theorem and the cluster algebra structure. A similar remark will come up again in the next section.
\end{rema}\fussy

\begin{theo}\label{diffGr} Let $\partial_t:f(t)\mapsto f'(t)\in\End(\C[t])$ and $T:f(t)\mapsto tf(t)\in\End(\C[t])$. Then
    $$h_{\C_{\hom}[Gr(d,n+d+1)]}(t)=\frac{(\partial_t^d\circ T^{d-1})^n\left[(1-t)^{-(d+1)}\right]}{\prod_{i=1}^d \prod_{j=1}^n(i+j)}.$$
\end{theo}
\begin{proof}
    We prove the theorem by induction on $n$. For $n=0$, we have $Gr(d,d+1)\cong\mathbb{P}^{d}$, so $$\mathrm{LHS}=h_{\C[x_0,x_1,\cdots,x_d]}(t)=\displaystyle\frac{1}{(1-t)^{d+1}}=\mathrm{RHS}.$$
    Suppose the statement holds for the case $(n-1)$ for some $n\geq1$. By applying the induction hypothesis, the formula (\ref{Vkomega}), and the same formula with $n$ replaced by $(n-1)$, we have
    \begin{align*}
        \mathrm{RHS}& =\frac{(\partial_t^d\circ T^{d-1})h_{\C[X(d,n+d)]}(t)}{\prod_{i=1}^d(n+i)}=\frac{\sum_{k=0}^{\infty}\prod_{i=0}^{d-1}\binom{k+n+i}{n}\binom{n+i}{n}^{-1}(\partial_t^d\circ T^{d-1})t^k}{\prod_{i=0}^{d-1}(n+1+i)}\\
        & =\sum_{k=1}^{\infty}\prod_{i=0}^{d-1}\frac{1}{n+1}\binom{k+n+i}{n}\binom{n+1+i}{n+1}^{\!\!-1}\prod_{j=1}^d(k+d-j) t^{k-1}\\
        & =\sum_{k=0}^{\infty}\prod_{i=0}^{d-1}\frac{1}{n+1}\binom{k+1+n+i}{n}\binom{n+1+i}{n+1}^{\!\!-1}\prod_{j=0}^{d-1}(k+1+j) t^{k}\\
        & =\sum_{k=0}^{\infty}\prod_{i=0}^{d-1}\binom{k+(n+1)+i}{n+1}\binom{(n+1)+i}{n+1}^{-1}t^{k}
        =\mathrm{LHS}.\qedhere
    \end{align*}
\end{proof}
Then we have an immediate corollary of Theorems \ref{Gr} and \ref{diffGr}.
\begin{coro}
    The $d$-dimensional Narayana polynomial
    \begin{align*}
        N_{d,n}(t)\coloneqq \sum_{k=0}^{(d-1)(n-1)}N_{d,n,k}t^k
    \end{align*}
  equals to
    \begin{align*}
        \frac{(1-t)^{dn+1}(\partial_t^d\circ T^{d-1})^{n-1}\left[(1-t)^{-(d+1)}\right]}{\prod_{i=1}^d \prod_{j=1}^{n-1}(i+j)}.
    \end{align*}
\end{coro}
\begin{rema}
Although the proof is independent, this result seems to be closely related to Remark 3.3 and calculations in Table A.7 in \cite{Agapito}. 
\end{rema}

\section{The Hilbert Series of $\C[\Ominbar^{A_n}]$.}
Let $\Omin^{A_n}$ be the minimal (nonzero) nilpotent orbit in $\slnn(\C)$, so $\Omin^{A_n}$ consists of nilpotent elements of rank one in $\slnn(\C)$. Since there is a natural scaling $\C^*$-action on $\Ominbar^{A_n}$, its coordinate ring $\C[\Ominbar^{A_n}]$ is a graded ring, and the goal of this section is to compute its Hilbert series.

We first recall the identity of Li Shan-lan, which has numerous proofs, and here we present its first known (but very difficult to find in the literature) proof due to Paul Tur\'an \cite{Turan}.
\begin{theo}[Li Shan-lan's Identity]\label{Li} For $n,m\in\N$,
$$
    \binom{m+n}{n}^2=\sum_{i=0}^{n}\binom{n}{i}^2\binom{2n+m-i}{2n}.
$$
\end{theo}
\begin{proof}
    Let \begin{equation}
           V_n(x)\coloneqq\left.-\frac{1}{n!}\frac{d^n}{dz^{n}}\left(\frac{z^n}{(z+x)^{n+1}}\right)\right|_{z=-1}. 
    \end{equation}
    Then 
    \begin{align*}
        V_n(x)& =-\frac{1}{n!}\left.\frac{d^n}{dz^{n}}\left(\frac{1}{z}\left(1+\frac{x}{z}\right)^{-n-1}\right)\right|_{z=-1}\\
        &=-\frac{1}{n!}\left.\frac{d^n}{dz^{n}}\sum_{k=0}^\infty\binom{-n-1}{k}\frac{x^k}{z^{k+1}}\right|_{z=-1}\\
        &=-\frac{1}{n!}\sum_{k=0}^{\infty}(-1)^k\binom{n+k}{k}x^k\left.\frac{d^n}{dz^{n}}\left(z^{-k-1}\right)\right|_{z=-1}\\
        &=\sum_{k=0}^\infty\binom{n+k}{k}^2x^k.
    \end{align*}
    On the other hand, we can differentiate the RHS of (5.2) by Leibniz's rule directly and get
    \begin{align*}
        V_n(x)& =-\sum_{k=0}^n\binom{n}{k}\binom{2n-k}{k}(x-1)^{-2n+k-1}\\
        & =\left(\sum_{k=0}^n\binom{n}{k}\binom{n+k}{n}(x-1)^{n-k}\right)(1-x)^{-2n-1}\\
        & =(1-x)^n P_n\left(\frac{1+x}{1-x}\right)\sum_{k=0}^\infty\binom{2n+k}{2n}x^k,
    \end{align*}
where $P_n(t)$ is the $n$-th Legendre polynomial and the third equality is proved by checking $$P_n(t)=\sum_{k=0}^n\binom{n}{k}\binom{n+k}{n}(-1)^{n-k}\left(\frac{t+1}{2}\right)^k$$ satisfies the characterizing ODE for Legendre polynomial $$(1-t^2)P_n''(t)-2tP_n'(t)+n(n+1)P_n(t)=0$$ with conditions $P_n(-1)=(-1)^n$ and $P'_n(-1)=\frac{n(n+1)}{2}(-1)^{n+1}.$

By Hurwitz formula for Legendre polynomials
$$
    (1-x)^nP_n\left(\frac{1+x}{1-x}\right)=\sum_{i=0}\binom{n}{i}^2x^i.
$$
So 
\begin{align*}
    V_n(x)& =\sum_{i=0}^{n}\binom{n}{i}^2x^i\ \sum_{k=0}^\infty\binom{2n+k}{2n}x^k \\
    & = \sum_{m=0}^\infty\left(\sum_{i=0}^m\binom{n}{i}^2\binom{2n+m-i}{2n}\right)x^m
\end{align*}
Compare the coefficients of $x^m$ in the first calculation of $V_n(x)$ at the beginning of the proof, we have the desired result. (The upper limits of the sum in the RHS of Li Shan-lan's  identity can be chosen freely between $n$ and $m$, as $\binom{n}{i}=0$ for $i>n$ and $\binom{2n+m-i}{2n}=0$ for $i>m$.)
\end{proof}

Let $n\in\N$ and take $\Lieg=\slnn(\C)$. We fix the standard choice of simple roots $\Pi=\{\alpha_1,\alpha_2,\cdots,\alpha_{n}\}$ and let $\omega_1,\cdots,\omega_{n}$ denote the corresponding fundamental weights.
Let $\rho$ be the half sum of the positive roots $\Phi^+\subset\Lieh^*$. Let $\theta\coloneqq\alpha_1+\cdots+\alpha_n$ be the highest root and $0\neq v_\theta\in\Lieg$ be any highest root vector. Then $V_\theta=\Lieg$ is the adjoint representation, and the closure $\Ominbar^{A_n}$ is the highest weight variety $X_{\theta}$.

\begin{lemm}\label{dimrep2} Let $k\in\N$. Then
    $$\dim(V_{k\theta})=\sum_{i=0}^n\binom{n}{i}^2\binom{k-i+2n-1}{2n-1}$$
\end{lemm}
\begin{proof} We know the highest root $\theta=\omega_1+\omega_n$. Let $k\in\N$. By Weyl's dimension formula, 
\begin{align*}
   \dim(V_{k\theta})& =\prod_{\alpha\in\Phi^+}\frac{\langle k\theta+\rho,\alpha\rangle}{\langle\rho,\alpha\rangle}=\prod_{\alpha\in\Phi^+}\frac{\langle k(\omega_1+\omega_n)+\omega_1+\cdots+\omega_n,\alpha\rangle}{\langle\omega_1+\cdots+\omega_n,\alpha\rangle}\\
   & =\left(\frac{k+1}{1}\cdot\frac{k+2}{2}\cdots\frac{k+n-1}{n-1}\right)^2\cdot\frac{2k+n}{n}\\
   & =\binom{k+n-1}{n-1}^2\cdot\frac{(n+k)^2-k^2}{n^2}\\
   & =\binom{k+n}{n}^2-\binom{k-1+n}{n}^2.
\end{align*}
Now we apply the identity of Li Shan-lan (Theorem \ref{Li}) and get
\begin{align*}
\dim(V_{k\theta})
&=\sum_{i=0}^n\binom{n}{i}^2\left(\binom{k+2n-i}{2n}-\binom{k-1+2n-i}{2n}\right) \\
&=\sum_{i=0}^n\binom{n}{i}^2\binom{k-i+2n-1}{2n-1}.\qedhere
\end{align*}
\end{proof}

Recall Definition \ref{Cat}.
In type $B$, we have
$$
    \Cat_{B_n}\coloneqq\prod_{i=1}^{n}\frac{(2i-1)+2n+1}{(2i-1)+1}=\binom{2n}{n}.
$$
In type $G_2$, we have
$$
    \Cat_{G_2}=\left(\frac{1+6+1}{1+1}\right)\left(\frac{5+6+1}{5+1}\right)=8.
$$
    There are numerous equivalent definitions of the (generalized) Narayana numbers; we refer to Theorem 5.9 in \cite{FominReading} and references therein. 
    Narayana numbers of $B$-type are
$$
    \displaystyle N_{B_n}(k)\coloneqq\binom{n}{k}^2.
$$
Then we can check
$$
    \sum_{k=0}^n N_{B_n}(k)=\sum_{k=0}^n\binom{n}{k}\binom{n}{n-k}=\binom{2n}{n}=\Cat_{B_n}.
$$
Narayana numbers of type $G_2$ are 
$
    N_{G_2}(0)=1, N_{G_2}(1)=6, N_{G_2}(2)=1,
$
which sum to $\Cat_{G_2}$.

\begin{theo}\label{Omin}
Let $\Ominbar^{A_n}$ be the minimal (nonzero) nilpotent orbit in $\slnn$. Then the Hilbert series for the graded algebra $\C[\Ominbar^{A_n}]$ is
    $$
        h_{\C[\Ominbar^{A_n}]}(t)=\frac{\sum_{i=0}^{n}\binom{n}{i}^2t^i}{(1-t)^{2n}}.
    $$
Note that the coefficients of the numerator are precisely Narayana numbers of type $B_n$ (or equivalently of type $C_n$ as they have the same Weyl group).
\end{theo}
\begin{proof}
    Recall that $\Ominbar^{A_n}$ is the highest weight variety $X_{\theta}$. Since the adjoint representation $\Lieg$ is identified with $\Lieg^*$ via the Killing form, we have $\overline{\theta}=\theta$. So by Corollary \ref{Hilbert} and Lemma \ref{dimrep2} we have
    \begin{align*}
        h_{\C[\Ominbar^{A_n}]}(t)& =\sum_{k=0}^{\infty}\dim(V_{k\theta})t^k
        =\sum_{k=0}^{\infty}\sum_{i=0}^n\binom{n}{i}^2\binom{k-i+2n-1}{2n-1}t^k\\
        & =\sum_{i=0}^{n}\binom{n}{i}^2t^i\ \sum_{j=0}^{\infty}\binom{j+2n-1}{2n-1}t^{j}=\frac{\displaystyle\sum_{i=0}^{n}\binom{n}{i}^2t^i}{(1-t)^{2n}}.\qedhere
    \end{align*}
\end{proof}
\begin{conj}
The coordinate ring $\C[\Ominbar^{A_n}]$ has a degenerate version (by setting some frozen variables to zero) of a cluster algebra structure of type $C_n$.
\end{conj}
\begin{rema} 
    In the special case where $n=2$ the above result said
    $$
        h_{\C[\Ominbar^{A_2}]}(t)=\frac{1+4t+t^2}{(1-t)^4}.
    $$
    We can do an analogous computation for $\Lieg=\mathfrak{so}(5,\C)$, so $\theta=\varepsilon_1+\varepsilon_2, \rho=\frac{1}{2}(3\varepsilon_1+\varepsilon_2)$ and
    \begin{align*}
        h_{\C[\Ominbar^{B_2}]}(t)
        &=\sum_{k=0}^{\infty}\dim(V_{k\theta})t^k \\
        &=\sum_{k=0}^{\infty}\prod_{\alpha\in\Phi^+}\frac{\langle k\theta+\rho,\alpha\rangle}{\langle\rho,\alpha\rangle}t^k \\
        &=\sum_{k=0}^{\infty}\binom{2k+3}{3}t^k
        =\frac{1+6t+t^2}{(1-t)^4}.
    \end{align*}
The numbers shown in the coefficients of the numerator are precisely Narayana numbers of type $G_2$. A similar question could be asked: does the coordinate ring of the minimal nilpotent orbit $\C[\Ominbar^{B_2}]$ have a degenerate version of a cluster algebra structure of type $G_2$?
\end{rema}

% The next command determines the bibliography style. Please do not
% change this.
\begingroup
\sloppy
\setlength{\emergencystretch}{3em}
\raggedright
\bibliographystyle{plain}
%  This inserts the bib file
\bibliography{samplebib}
\endgroup

\bigskip
\noindent Boming Jia \\Yau Mathematical Sciences Center,\\ Tsinghua University, Beijing 100084, China\\
jiabm@tsinghua.edu.cn

\end{document}